\documentclass[12pt]{article}
\usepackage[utf8]{inputenc}
\usepackage[margin=1in]{geometry}
\usepackage{amsmath}
\usepackage{mathrsfs}
\usepackage{amssymb}
\usepackage{amsfonts}
\usepackage{mathtools}
\usepackage{amsthm}
\usepackage{tikz-cd}
\usepackage{thmtools}
\usepackage{accents}
\usepackage{thm-restate}
\usepackage{hyperref}
\usepackage{CommonCommand}
\usepackage{subfiles} % Best loaded last in the preamble

\hypersetup{colorlinks=true,linkcolor=blue, linktocpage}

\theoremstyle{plain}
\newtheorem{thm}{Theorem}[section]
\newtheorem{lemma}[thm]{Lemma}
\newtheorem{cor}[thm]{Corollary}

    \newtheoremstyle{TheoremNum}
        {\topsep}{\topsep}              %%% space between body and thm
        {\itshape}                      %%% Thm body font
        {}                              %%% Indent amount (empty = no indent)
        {\bfseries}                     %%% Thm head font
        {.}                             %%% Punctuation after thm head
        { }                             %%% Space after thm head
        {\thmname{#1}\thmnote{ \bfseries #3}}%%% Thm head spec
    \theoremstyle{TheoremNum}
    \newtheorem{thmn}{Theorem}
    \newtheorem{coron}{Corollary}

\theoremstyle{definition}
\newtheorem{defn}[thm]{Definition}
\newtheorem{exmp}[thm]{Example}
\newtheorem{rmk}[thm]{Remark}

    \newtheoremstyle{DefNum}
        {\topsep}{\topsep}              %%% space between body and thm
        {}                              %%% Thm body font
        {}                              %%% Indent amount (empty = no indent)
        {\bfseries}                     %%% Thm head font
        {.}                             %%% Punctuation after thm head
        { }                             %%% Space after thm head
        {\thmname{#1}\thmnote{ \bfseries #3}}%%% Thm head spec
    \theoremstyle{DefNum}
    \newtheorem{defnn}{Definition}

\newcommand{\K}{\bb{K}}

\renewcommand{\O}{\mathcal{O}}
\newcommand{\mf}[1]{\mathfrak{#1}}
\newcommand{\m}{\mf{m}}
\newcommand{\len}{\lambda}
\DeclareMathOperator{\Spec}{Spec}
\DeclareMathOperator{\Supp}{Supp}
\renewcommand{\ring}[1]{\mathring{#1}}

\title{Generalization of Gurjar's Hyperplane section Theorem to arbitrary analytic varieties and A$\m$AC classes}
\author{A.J. Parameswaran, Mohit Upmanyu}

\begin{document}

\maketitle

\begin{abstract}
This paper aims to generalize the hyperplane section Theorem of R.V. Gurjar to arbitrary (local) analytic varieties even if the intersection with hyperplanes is not necessarily isolated. 

In the case of formal varieties, we generalize the statement to work for different classes of hypersurfaces other than hyperplanes. We call the classes of functions (which are subsets of the formal power series ring) defining these classes of hypersurface algebraic $\m$-adically closed (A$\m$AC) classes.
\end{abstract}

\tableofcontents

\section{Introduction}

Let $A = \bb{C}[\![x_1,x_2,\ldots,x_n]\!]$ and $\m$ denote its maximal ideal. 
To illustrate the problem under consideration we begin with an example. 
We look at quadratic singularity defined by the equation $f := x_1^2+\cdots+x_n^2$.
Let $X = \Spec(A/\<f\>)$ and further, let $h := x_1 + i x_2 \in \m - \m^2$ and $H$ be the subscheme defined by $h$.
Now, one can see that $X \cap H$ has a line singularity.
\\

Let $h_j := x_1 + i x_2  + x_2^j \in \m - \m^2$ and let $H_j$ be the subscheme defined by $h_j$.
We can regard $H_j \iso \Spec(\bb{C}[\![x_2,\ldots,x_n]\!])$, then $X \cap H_j$ considered as a subscheme of $\Spec(\bb{C}[\![x_2,\ldots,x_n]\!])$ is given by the equation $2ix_2^{j+1} - x_2^{2j} + x_3^2+\cdots+x_n^2$. The Milnor number of $X \cap H_j$ is $j$.
\\

Following the example above it is possible to prove:

\begin{thm}(\cite{GurjarHST})%[Section 2, Converse]
Let $A = \bb{C}[\![x_1,x_2,\ldots,x_n]\!]$ and $\m$ denote its maximal ideal.
Let $f \in \m^2$ \st the subscheme $X \subset \Spec(A)$ defined by $f$ has a curve singularity. Assume $h \in \m-\m^2$ \st $X \cap H$ has a non-isolated singularity. 
Then there exists an infinite sequence $h_1,\dots,h_i,\dots \in \m-\m^2$ \st $X \cap H_i$ has an isolated singularity and the Milnor number of $X \cap H_i \to \infty$ as $i \to \infty$.
\end{thm}

One can then ask whether the converse holds. In answer to this, R.V. Gurjar proved the following hyperplane section theorem.

\begin{thm}\label{GHST}(\cite{GurjarHST})
Let $A = \bb{C}[\![x_1,x_2,\ldots,x_n]\!]$ and $\m$ denote its maximal ideal.
Let $f \in \m^2$. Assume that for every $h \in \m-\m^2$,
the ring $A/\<f,h\>$ is an isolated singularity of dimension $n-2$. 
Then there abstractly\footnote{We use the word abstractly to indicate that what we give is an existential argument. More precisely, what we prove is that $\mu$ is not infinite so we conclude that a bound exists.} exists an integer $N(f)$ (depending on f) such that $\mu(A/\<f, h\>) \leq N(f)$, for any $h \in \m-\m^2$.
\end{thm}

Here $\mu(A/\<f, h\>)$ is the Milnor number of the hypersurface defined by $f,h$.
\\

Later Shigefumi Mori gave a simpler proof of \ref{GHST} in \cite{MoriHST}, where to prove the non-emptiness of a certain set, he critically used the fact that it is the inverse limit of Zariski constructible sets. Inspired by this, we give the following definition.

\newcommand{\AmACDefn}{Let $\K$ be an algebraically closed uncountable field. Let $A = \K[\![x_1,x_2,\ldots,x_n]\!]$ and $\m$ denote its maximal ideal.
A subset $P \subset B = \K[\![x_1,\ldots,x_n]\!]/I$ is said to be an \textbf{algebraic $\m$-adically closed (A$\m$AC) class}, if $P = \ula{\lim} P_i$ where $P_i$ are constructible subsets of $B/\m^i$ (in Zariski topology).
Since $P$ is an inverse limit, $P$ is automatically $\m$-adically closed (see \ref{inverse limit criterion}).}

\begin{defnn}[\ref{AmAC defn}]
\AmACDefn
\end{defnn}

In November 2018 the first author along with Mihai Tibar (during the Research in Pairs program at Mathematisches Forschungsinstitut Oberwolfach) had proved a version of Gurjar's Theorem \ref{GHST}. In the result they looked at hypersurface sections involving classes of hypersurfaces with bounded Milnor number.
\\

This paper aims to extend this result to arbitrary formal varieties and hypersurface sections defined by A$\m$AC classes of functions. Furthermore, we also extend the result to work for hypersurfaces sections with non-isolated singularities. This goal is achieved by using the multiplicity of the singular locus instead of the Milnor number $\mu$.

We also extend the results to the analytic case (over $\bb{C}$) if the A$\m$AC Classes are finitely determined.

If $Y$ is an equidimensional closed subvariety of $\Spec(A)$ then $\S_{Y}$ denotes the singular locus of $Y$ considered as a scheme (see definition \ref{singular Definition}). We can now state the main result.

\newcommand{\AmACResult}{Let $X$ be an equidimensional closed subvariety of $\Spec(A)$. Let $ZD(X)$ denote the zero divisors of $\O_X$ in $A$. Let $P \subset A$ be an algebraic $\m$-adically closed class \st $ZD(X) \cap P = \emptyset$. Let $h_i \in P$ be an infinite sequence \st $\S_{X \cap H_i}$ is of dimension $d$ and the multiplicity of $\S_{X \cap H_i} \to \inf$ as $i \to \inf$. Then there exists $h \in P$ \st $\S_{X \cap H}$ is of dimension $\geq d+1$.}

\begin{thmn}[\ref{AmAC result}]
\AmACResult
\end{thmn}

Let us now put $P = \m-\m^2$ and take $X$ to be an irreducible variety not embedded in any hyperplane (to get rid of the hypothesis of zero divisors).

\newcommand{\HyperplaneSectionTheorem}{Let $X$ be an irreducible closed subvariety of $\Spec(A)$ and further assume $X$ is not embedded in any hyperplane. Let $h_i \in \m-\m^2$ be an infinite sequence \st $\S_{X \cap H_i}$ is of dimension $d$ and the multiplicity of $\S_{X \cap H_i} \to \inf$ as $i \to \inf$. Then there exists $h \in \m-\m^2$ \st $\S_{X \cap H}$ is of dimension $\geq d+1$.}

\begin{coron}[\ref{Hyperplane section theorem}]
\HyperplaneSectionTheorem
\end{coron}

\begin{rmk}
Other interesting examples of A$\m$AC which are suitable for Theorem \ref{AmAC result} include
\begin{itemize}
\item $\{h \in A \ |$ Milnor number of $h$ is constant\} (example \ref{basic example}.6)
\item $\{h \in A \ |$ Tjurina number of $h$ is constant\} (example \ref{basic example}.4)
\item $\{h \in A \ | \ \dim(\S_{H}) = 1$, the multiplicity and the $\m$-torsion of $\S_{H}$ are both constant\} (example \ref{advanced example}.9)
\end{itemize}

Finally, we note that A$\m$AC classes are closed under finite union and countable intersection (see Lemma \ref{basic lemma}). For example, $\{h \in A \ | \ \text{Milnor number of $h$ is bounded} \}$ can also be used in Theorem \ref{AmAC result}. 
\end{rmk}

We will now illustrate that Theorem \ref{AmAC result} is indeed a generalization of Gurjar's Theorem.
\\

Let $f \in \m^2 \subset A$ satisfying hypothesis of Gurjar's Theorem i.e. $\mu(A/\<f, h\>)$ is finite for all $h \in \m - \m^2$. Let $X$ be the subscheme defined by $f$.

Assume to the contrary that $\mu(X \cap H)$ is unbounded. Then from \cite{liu2018milnor}, we know that $\mu(X \cap H) \leq (\dim(X \cap H) + 1)\tau(X \cap H)$.
We now see that Tjurina numbers $\tau(X \cap H)$ must also be unbounded. 
But the Tjurina number is also the multiplicity of the zero dimensional singular locus.
So we can obtain a sequence of hyperplanes $h_i$ \st $\S_{X \cap H_i}$ is of dimension $0$ and the multiplicity of $\S_{X \cap H_i} \to \inf$ as $i \to \inf$. 

We now apply Corollary \ref{Hyperplane section theorem} to get a hyperplane $h$ \st $X \cap H$ has a non-isolated singularity. This is a contradiction, so $\mu(X \cap H)$ must be bound.
\\

We have also obtained a more specialized result for surfaces. 

\newcommand{\surfaceResult}{Let $X \subset \Spec(B) (= \Spec(A/I))$ be an equidimensional closed subvariety of dimension $2$. Let $ZD(X)$ denote the zero divisors of $\O_X$ in $B$. Let $P \subset B$ be an algebraic $\m$-adically closed class \st $ZD(X) \cap P = \emptyset$. Then there abstractly exist constants $e(P)$ and $t(P)$ \st for every hypersurface $h \in P$, multiplicity tuple of $\S_{H}$ is bounded by $(e(P),t(P))$ in the \textbf{lexicographic ordering} and there is a hypersurface whose singular locus achieves this maximum.}

\begin{thmn}[\ref{surface result}]
\surfaceResult
\end{thmn}
The multiplicity tuple of dimension $1$ local ring $R$ is defined as $$\text{(multiplicity of }R\text{ w.r.t. }\m \ , \m\text{-torsion in }R)$$
So the above result can also be understood as saying: ``From Theorem \ref{AmAC result} we already know that the multiplicity of singular locus of $X \cap H$ is bounded. But among the hypersurfaces of $P$ with maximum multiplicity, the amount of $\m$-torsion in $\O_{\S_H}$ is also bounded."
\\

We also use A$\m$AC classes to prove the following result:

\newcommand{\milnorTjurina}{Let $X \subset \Spec(A)$ be an isolated hypersurface singularity with Tjurina number $\tau$. Then there abstractly exists a constant $c>0$ \st $\mu(X) \leq c$. This $c$ depends only on $\tau$ and $\dim(X)$.}

\begin{thmn}[\ref{milnor tjurina}]
\milnorTjurina
\end{thmn}

In \cite{DimcaGreuel}, $\mu \leq \frac{4}{3}\tau$ was proposed for the case of plane curves. This was independently proven in \cite{2020minimal} and \cite{2019saitos} for the case of irreducible plane curves. In \cite{2019Almiron}, the bound for all plane curves was proven. Moreover, in \cite{2019Almiron}, it was proven that if the Durfee conjecture holds for a hypersurface $X$ of dimension 2, then $\mu \leq \frac{3}{2}\tau$.
\\

For the general case, an explicit bound of $\mu \leq (\dim(X)+1)\tau$ has been proven in \cite{liu2018milnor}.
In a later work, we have also extended the result \ref{milnor tjurina} to the case when $X$ is an isolated complete intersection singularity. (\cite{PU2023})
\\

The paper is organized as follows:

In Section 2, we prove essential results concerning the multiplicity and length of rings, as well as preliminary results necessary for working with the singular locus and proving the constructiblity of subsets of $A/\m^k$ in the Zariski topology.

In Section 3, we define and prove basic properties of A$\m$AC classes of functions, along with examples. We then utilize these examples to prove our main results (Theorems \ref{AmAC result} and \ref{Hyperplane section theorem}). We also use A$\m$AC to prove that the Milnor number of a hypersurface is bounded with a bound dependent only on its dimension and Tjurina number (Theorem \ref{milnor tjurina}).

In Section 4, we specialize our A$\m$AC to curve singularities and utilize them to prove a more specialized result for surfaces (Theorem  \ref{surface result}).

In Section 5, we define finitely determined A$\m$AC. We prove a version (Theorem \ref{AmAC result}) for analytic varieties by restricting ourselves to finitely determined A$\m$AC which allows us to apply Artin's approximation theorem. We also prove a converse of the Corollary \ref{Hyperplane section theorem}.

%\newpage

\section{Preliminaries}

Let $R$ be a Noetherian ring of dimension $d$ with a maximal ideal $\m$. Let $I$ be an $\m$-primary ideal. We will   denote the length of a module $M$ as $\len(M)$.

\begin{defn}\label{multi definition}
Define the Hilbert-Samuel function $\chi_{R}^{I}(n) := \len(R/I^{n+1})$.
If the maximal ideal is clear from the context and $I=\m$, then we will drop the $\m$ in the notation.
We take the definition of multiplicity as the following 
$$e(R,I) = \lim_{n\to\inf} d! \frac{\len(R/I^{n+1})}{n^d} $$
\end{defn}

We will use the following fundamental result from dimension theory.
\begin{lemma}\label{dimension}
Let $R$ be a Noetherian local ring with Krull dimension $d$ and maximal ideal $\m$. Let $I$ be an $\m$-primary ideal.
Then for sufficiently large $n$ we have $\chi^{I}_{R}(n) = p(n)$, where $p(n)$ is a rational polynomial of degree $d$ (called the Hilbert-Samuel polynomial). 
\end{lemma}
\begin{proof}\cite[Theorem 13.4]{MatsumuraCRT}.\end{proof}

\begin{lemma}\label{poly degree}
Let q(x) be a rational polynomial of degree $d$ \st $q(x) > 0$ for sufficiently large $x$ i.e it has a positive highest coefficient. Let p(x) be another rational polynomial \st $p(x) \geq q(x)$ for infinitely many positive integers. Then $\deg(p)\geq \deg(q)=d$ and $p(x) > 0$ for $x$ sufficiently large.

Furthermore, if $\deg(p) = \deg(q)$, then the highest coefficient of $p(x)$ is greater than or equal to the highest coefficient of $q$.
\end{lemma}
\begin{proof}
We prove this by contradiction. Let $\deg(p) < d$. Then for sufficiently large $n$, $p(n) < q(n)$, which is a contradiction to the assumption. 

Similarly, if $\deg(p)=\deg(q)$ and the highest coefficient of $p(x)$ is less than the highest coefficient of $q(x)$, then again for sufficiently large $n$, we have the contradiction $p(n) < q(n)$.
\end{proof}

The next Lemma will be used in the inductive step of the Lemma \ref{dim bound lemma}. 

\begin{lemma}\label{inductive lemma}
Let $R$ be a Noetherian ring of dimension $d>0$ with a maximal ideal $\m$ \st $R/\m$ is infinite.
Then there exist two graded rings $\~{R}$ of dimension $d$ and $\-{R}$ of dimension $d-1$ \st

$$e(R,\m) = e(\~{R},\~{\m}) = e(\-{R},\-{\m})$$
$$ \chi_{R}^{\m}(n) \geq \chi_{\~{R}}^{\~{\m}}(n)$$
$$ \chi_{\~R}^{\~{\m}}(n) = \chi_{\~{R}}^{\~{\m}}(n-1) + \chi_{\-{R}}^{\-{\m}}(n) $$

Here $\~{\m}$ and $\-{\m}$ are the irrelevant ideals of the graded rings $\~{R}$ and $\-{R}$ respectively.

\end{lemma}

\begin{proof}
Let $gr_{\m}(R)$ denote the graded algebra $R/\m \oplus \m/\m^2 \oplus \cdots$ with the standard grading. (which we note is generated as an $R/\m$-algebra by the first grade $\m/\m^2$).

Since the Hilbert-Samuel function of $R$ and $gr_{\m}(R)$ are equal, without loss of generality, $R = gr_{\m}(R) = R_0 \oplus R_1 \oplus \cdots $. Therefore, $R$ is a graded $R/\m$-algebra generated by $R_1$.

Let $x \in R_1$ be a part of a system of parameters so $\dim(R/\<x\>) = d-1$.

Define $T_{x} := \{a \in R \ |\ \exists k \text{ \st } x^{k} a = 0\}$, which is a graded ideal (as $x^k$ are homogeneous elements).
Define $\~{R} := R/T_{x}$. This is a graded ring and $\~{x}$ (image of $x$ in $\~{R}$) is a nonzero divisor.
Since $\~{R}$ is a quotient of $R$, we get $\chi_R(n) \geq \chi_{\~{R}}(n)$. 
Since $\Supp(T_{x}) \subset V(x)$ is of dimension strictly less than $d$, $e(R,\m) = e(\~{R},\~{\m})$.

Define $\-{R} := \~{R}/\<\~{x}\>$. Since $\~{x}$ is a nonzero divisor, it is also a superficial element in $\~{R}$ (As in \cite{Sam51}). Now by \cite[pg 188, Application to hibert function]{Bondil2005}, we see that the equality we need does indeed hold for Hilbert functions, which we could then use to derive our result. But for the sake of completion we instead derive it here.

Since $\~{x}$ is a nonzero divisor, $\<\~{x}\>$ is isomorphic to $\~{R}$, as $\~{R}$ modules, and $\dim(\-{R}) = \dim(R) - 1$. 

\begin{align} \nonumber
\begin{split}
\len(\frac{\-{R}}{\m^{r+1}}) = \len(\frac{\~{R}}{\<\~{x}\>+\m^{r+1}}) & = \len(\frac{\~{R}}{\m^{r+1}}) - \len(\frac{\<\~{x}\>+\m^{r+1}}{\m^{r+1}})\\
& = \len(\frac{\~{R}}{\m^{r+1}}) - \len(\frac{\<\~{x}\>}{\<\~{x}\> \cap \m^{r+1}})\\
& = \len(\frac{\~{R}}{\m^{r+1}}) - [\len(\frac{\<\~{x}\>}{\~{x}\m^{r}}) - \len(\frac{\<\~{x}\> \cap \m^{r+1}}{\~{x}\m^{r}})]
\end{split}
\end{align}
Since $\~{R}$ is a graded ring and $\~{x}$ is a nonzero divisor (and a homogeneous element of degree 1), $\<\~{x}\> \cap \~{R}_{r+1} = \<\~{x}\>_{r+1} = \~{x}\~{R}_{r}$ (where $\<\~{x}\>_{r+1}$ is the $(r+1)^{th}$ graded piece of ideal $\<\~{x}\>$). Then we have 
$$\<\~{x}\> \cap \m^{r+1} = \<\~{x}\> \cap (R_{r+1} \oplus R_{r+2} \oplus \cdots) = \~{x}R_{r} \oplus \~{x}R_{r+1} \oplus \cdots = \~{x}\m^{r}$$

Also we see that $R/\m^{r} \xra{\cross x} \<x\>/x\m^{r}$ is an isomorphism (as $\<\~{x}\> \iso \~{R}$). So we get,

\[\chi_{\-{R}}(r) = \chi_{\~{R}}(r) - \chi_{\~{R}}(r-1).\]

Using the definition of multiplicity (or \cite[Theorem 14.11]{MatsumuraCRT}) $$e(R,\m) = e(\~{R},\~{\m}) = e(\-{R},\-{\m})$$
\end{proof}

\begin{rmk}
The inequality (1) and (2) of the following Lemma are used in the proof of the main Theorems, while (3) will be used only in Section 5 to give examples (and hence can be safely skipped).
\end{rmk}

\begin{lemma}\label{dim bound lemma}
%Let $S_d$ denote a regular local ring of dimension $d$.
Let $R$ be a Noetherian ring of dimension $d$ with a maximal ideal $\m$ \st $R/\m$ is infinite. Then 
\begin{equation}\chi_R^{\m}(n) \geq {n+d\choose d}\end{equation}
%= \chi_{S_d}(n)
Let $e := e(R,\m)$. Then we can obtain better bounds as follows:
\\

For $n \leq e-1$,
\begin{equation}\chi_R^{\m}(n) \geq {n+d+1\choose d+1} \end{equation} %= \chi_{S_{d+1}}(n)
For $n \geq e - 1$,
\begin{equation}\chi_R^{\m}(n) \geq \sum_{i=0}^{\min(e,d)} (-1)^{i} {e \choose i+1}{n+d-i\choose d-i} = e{n+d\choose d} + \text{lower order terms (of $n$)} \end{equation}
\end{lemma}
Note that ${n+d\choose d}$ is the Hilbert-Samuel polynomial of a regular local ring of dimension $d$.

\begin{proof}
It can be easily seen that (1) is a special case of (3) at $e = 1$, but we give a more geometric proof. A proof of (1) can be found in \cite[Corollaire 2 pg.68]{BourbakiCA8-9}.

Since the Hilbert-Samuel function of $R$ and $gr_{\m}(R)$ are equal, without loss of generality, we may assume that $R = R_0 \oplus R_1 \oplus \cdots $ is a graded ring generated by $R_1$.

Now apply Noether normalization on $\text{Proj} (R)$. 
This gives us a graded ring homomorphism from $A = \frac{R}{\m}[x_1,\ldots ,x_d] \to R$, which is also an integral and dominant map. In particular, the map is injective. So 
$$ {n+d\choose d} = \len\bigg(\frac{A}{m^{n+1}}\bigg) = \len(A_0 \oplus A_1 \oplus \cdots \oplus A_n) \leq \len(R_0 \oplus R_1 \oplus \cdots \oplus R_n) = \len\bigg(\frac{R}{\m^{n+1}}\bigg)$$

We prove (2) using induction on dimension.

\textbf{Base case :} $\dim(R)=0$.

We see that if $\m^{r}/\m^{r+1} = 0$ then by Nakayama Lemma, $\m^{r} = 0$ which implies $\len(R/\m^{r})$ is a strictly monotonic function of $r$ until $\len(R/\m^{r}) = e$. $$\chi_R^{\m}(n) \geq {n+0+1\choose 0+1} = n+1 \text{ (if } n < e \text{)}$$

Now we prove the \textbf{inductive step}, with $\dim(R) = d$. 
We use Lemma \ref{inductive lemma} to get, $\~{R}$ and $\-{R}$ \st $\-{R}$ has dimension $d-1$ and multiplicity $e$.
\begin{align} \nonumber
\begin{split}
\chi_{R}^{\m}(n) \geq \chi_{\~{R}}^{\~{\m}}(n) &= \sum_{r=0}^n \chi_{\-{R}}^{\-{\m}}(r)\\
&\geq \sum_{r=0}^n {r+d\choose d} \quad (\text{if} \  n < e) \quad \text{(using induction on $\-{R}$)} \\
& = {n+d+1\choose d+1} \quad (\text{if} \  n < e)
\end{split}
\end{align}

We prove (3) using double induction on $n$ and dimension $d$. Firstly, we note that the coefficient of $x^{d+1}$ in the product of the power series $(1+x)^{-e}$ and $(1+x)^e$ is $0$ (as dimension $=d\geq 0$). Therefore

\begin{equation} \tag{*} \sum_{i=0}^{\min(e,d+1)} (-1)^{i} {e \choose i}{(e-1)+d+1-i\choose d+1-i} = 0 \end{equation}

\textbf{Base Case :} $n= e-1$. Using (2)
\begin{align} \nonumber
\begin{split}
\chi_{\~{R}}^{\~{\m}}(e-1) & \geq {(e-1)+d+1\choose d+1}  \\
& = \sum_{i=0}^{\min(e,d)} (-1)^{i} {e \choose i+1}{(e-1)+d-i\choose d-i} \text{ (using *))}
\end{split}
\end{align}
\textbf{Inductive step:} 
We assume the inequality is true for all rings of dimension $< d$, and for rings of dimension $= d$, the inequality is true until $n-1$.
Using Lemma \ref{inductive lemma} construct $\~{R}$, $\-{R}$ as before.
Since $\dim(\-{R}) = d-1$ by the induction hypothesis, the inequality holds for all $n \geq e-1$. Since $\dim(\~{R}) = d$ the inequality holds until $n-1$.
\begin{align} \nonumber
\begin{split}
\chi_{R}^{\m}(n) \geq \chi_{\~{R}}^{\~{\m}}(n) &= \chi_{\~{R}}^{\~{\m}}(n-1) + \chi_{\-{R}}^{\-{\m}}(n) \\ 
&\geq \sum_{i=1}^{d} (-1)^{i} {e \choose i+1}{n-1+d-i\choose d-i} + \sum_{i=1}^{d-1} (-1)^{i} {e \choose i+1}{n+d-1-i\choose d-i-1}\\
&= \sum_{i=1}^{d-1} (-1)^{i} {e \choose i+1}\Bigg({n+d-i-1\choose d-i} + {n+d-i-1\choose d-i-1}\Bigg) + (-1)^{d} {e \choose d+1}{n-1\choose 0}\\
&= \sum_{i=1}^{d-1} (-1)^{i} {e \choose i+1}{n+d-i\choose d-i} + (-1)^{d} {e \choose d+1}.1\\
&= \sum_{i=1}^{d} (-1)^{i} {e \choose i+1}{n+d-i\choose d-i}
\end{split}
\end{align}
(Here we have assumed that $d \leq e$, the proof is similar for $e<d$).
\end{proof}

Fix $\K$ to be an uncountable algebraically closed field (eg. $\bb{C}$, $\-{\bb{Q}_p}$, $\-{\bb{F}_p(\!(t)\!)}$).
Fix $A = \K[\![x_1,\dots,x_n]\!]$ and fix $\m$ to be its maximal ideal.

\begin{defn}\label{singular Definition}
Let $X$ be a closed subvariety of $\Spec(A)$ and $I(X) = \<f_1,f_2,\dots,f_k\>$ be the ideal defining $X$. Define $\S^{\a}_{X}$ as the scheme of ``all points $p \in X$ \st $\dim(T_p(X)) > \a$", whose scheme structure is given as follows:

$$\S^{\a}_{X} := X \cap \Spec\frac{A}{J_{n-\a}(X)} = \Spec\frac{A}{I(X)+J_{n-\a}(X)}$$
where $J_{n-\a}(X)$ is the ideal generated by $(n-\a) \cross (n-\a)$ minors of the Jacobian matrix $(\frac{\doe f_i}{\doe x_j})$.

As seen in \cite[4.D]{LooijengaICIS}, $J_{n-\a}(X)$ corresponds to the Fitting ideal of the module $\Th_f := Coker(A^k \xra{Jac(f_i)} A^n)$, so the scheme $\S^{\a}_X$ does not depend on the choice of generators of the ideal.

If $X$ is equidimensional of dimension $\a$, then we denote the singular locus of $X$ as a scheme by $\S_{X}$ and is defined as follows:
$$\S_{X} := \S^{\a}_{X} = X \cap \Spec\bigg(\frac{A}{J_{n-\a}(X)}\bigg) = \Spec\bigg(\frac{A}{I(X)+J_{n-\a}(X)}\bigg)$$
As seen in \cite[4.A]{LooijengaICIS}, this scheme corresponds exactly to the singular locus of $X$.
\end{defn}

The next two results are semi-continuity results which will be used throughout this paper at various places to prove that certain sets are Zariski constructible.

\begin{lemma}\label{basic semicontinuity}
Let $B = \K[\![x_1,\ldots,x_n]\!]/I$ and $\m$ be its maximal ideal. Let $I_1,\dots,I_k$ be $\m$-primary ideals and $\~{I} := I_1+\cdots+I_k$. Define 
\[\len_{\~{I}} : B/I_1 \cross \cdots \cross B/I_k \to \bb{Z}_{\geq 0}\]
\[\len_{\~{I}}(y_1,\dots,y_k) = \len(\frac{B}{\~{I}+\<y_1,\dots,y_k\>})\]
Then $\len_{\~{I}}$ is upper semi-continuous, when $B/I_1 \cross \cdots \cross B/I_k$ is given the Zariski topology.
\end{lemma}
\begin{proof}
\[\len\bigg(\frac{B}{\~{I}+\<y_1,\dots,y_k\>}\bigg) = \len\bigg(\frac{B}{\~{I}}\bigg) -  \len\bigg(\frac{\~{I}+\<y_1,\dots,y_k\>}{\~{I}}\bigg)\]
So it is enough to prove that that $\len(\frac{\~{I}+\<y_1,\ldots,y_k\>}{\~{I}})$ is lower semi-continuous, since $\K$ is algebraically closed $\len = \dim_{\K}$. 

Define the linear map, $(B/\~{I})^k \to B/\~{I}$, given by $(a_i) \mapsto \sum_{i=i}^k y_ia_i$.
Note that the image is exactly $\frac{\~{I}+\<y_1,\dots,y_k\>}{\~{I}}$, therefore the rank of this map is exactly $\dim_{\K}(\frac{\~{I}+\<y_1,\dots,y_k\>}{\~{I}})$. But rank is lower semi-continuous (in the variable $(y_1,\dots,y_k)$) as needed.
\end{proof}

\textbf{Notation:} Let $B = \K[\![x_1,\dots,x_n]\!]/I$ and $h \in B$ then by $H$ we will denote the corresponding subvariety, i.e. $H = \Spec(B/\<h\>)$

\begin{lemma}\label{singular locus semicontinuity}
Let $B = \K[\![x_1,\dots,x_n]\!]/I$ be a ring of dimension $\a+1$ and $\m$ be its maximal ideal. Define
\[\len_{\S} : B/\m^r \to \bb{Z}_{\geq 0}\]
\[\len_{\S}(h) = \len(\O_{\S^{\a}_H}/\m^{r-1})\]
$\len_{\S}$ is well defined and upper semi-continuous, when $B/\m^r$ is given the Zariski topology 
\end{lemma}
\begin{proof}
Let $A = \K[\![x_1,\ldots,x_n]\!]$ and $I = \<f_1,\ldots,f_k\>$
Define \[
\begin{array}{ccccc}
A/\m^r &\to& (A/\m^r)\cross(A/\m^{r-1})^{{k+1 \choose \a}\times{n\choose \a}} &\to& \bb{Z}_{\geq 0} \\
\-h &\mapsto& (\-h, n-\a\cross n-\a \text{ minors of Jacobian of }(\-h,f_1,\ldots,f_k)) &\mapsto& \len(\frac{A}{\m^{r-1}+I+\<h\> + Jac_{n-\a}(H)})
\end{array}
\]
Choose $h_1,h_2 \in \m-\m^2$ \st $h_1 = h_2 = \-h$ modulo $\m^{r}$. Then $\frac{\doe h_1}{\doe x_i} = \frac{\doe h_2}{\doe x_i}$ modulo $\m^{r-1}$, so the first map is well defined. Note that the first map is a multilinear map on the vector space $A/\m^r$ hence continuous in the Zariski topology. And the second map is upper semi-continuous by previous Lemma \ref{basic semicontinuity}.
Now we go modulo $I$ to get the map $\len_{\S}:B/\m^r \to \bb{Z}_{\geq 0}$, which is also upper semi-continuous.
\end{proof}

\section{Algebraic $\m$-adically closed class}

Here we introduce the most important concept of algebraic $\m$-adically closed classes that are critical in proving the main Theorem \ref{AmAC result}.

\begin{defn}\label{AmAC defn}
\AmACDefn

If $U$ is a set complement of an A$\m$AC then we refer to it as an \textbf{algebraic $\m$-adically open (A$\m$AO) class}.
\end{defn}

\begin{lemma}\label{inverse limit criterion}
Let $P \subset B = \K[\![x_1,\ldots,x_n]\!]/I$. Then, the following are equivalent:
\begin{enumerate}
\item $P = \ula{\lim} P_i$, where $P_i \subset B/\m^i$.
\item $P$ is an $\m$-adically closed subset.
\end{enumerate}
\end{lemma}
\begin{proof}
(1) $\implies$ (2)

The $\m$-adic topology is given by $B = \ula{\lim} B/\m^i$, where $B/\m^i$ is given the discrete topology. But $P_i$ is closed in the discrete topology of  $B/\m^i$. So $P = \ula{\lim} P_i$ is $\m$-adically closed since it is the inverse limit of closed subsets.

(2) $\implies$ (1)

Let $\pi_i: B \to B/\m^i$ be the canonical projection map.
Let $P_i := \pi_i(P)$. Define $\-P := \ula{\lim} P_i$. 
As seen before, $\-P$ is $\m$-adically closed. We will prove that $\-P$ is the closure of $P$. Since $P$ is also $\m$-adically closed, we will obtain $P = \-P = \ula{\lim} P_i$.

Let $p \in \-P$, then $p = (\dots,\pi_i(p_i),\pi_{i+1}(p_{i+1}),\dots)$. and $p_i \in P$. The sequence $p_i$ converges to $p$ in the $\m$-adic topology (since $\forall i\ \exists p_i = p$ mod $\m^i$). Hence $\-P$ is the closure of $P$.  
\end{proof}

\begin{rmk}
Note that $P_i$ in general are \textbf{neither} closed \textbf{nor} open in the Zariski topology, but are instead closed in discrete topology (since every subset is closed in discrete topology).
\end{rmk}

In \cite{MoriHST}, Mori proves the non-emptiness of a certain fixed subset of $\bb{C}[\![x_1,\dots,x_n]\!]$ by proving that it is an inverse limit of the Zariski constructible set. We in turn have turned this into the definition of an A$\m$AC class. The following \textbf{non-emptiness lemma} is unsurprising. 

\begin{lemma}\label{nonempty lemma}
Let $P = \ula{\lim} P_i$ be an A$\m$AC class. $P \neq \emptyset$ iff $P_i \neq \emptyset \ \forall i$.
\end{lemma}
\begin{proof}
We have a sequence of maps $\cdots \to P_{r+1} \to P_r \to P_{r-1} \to \cdots$ given by the restriction of projection maps.

Define $W_{r,s} :=$ closure of $image(P_r \to P_s)$. (Closure is taken in $P_s$ w.r.t. Zariski topology).

Define $W_{s} := \bigcap_{r>s} W_{r,s}$. 

This is an intersection of closed subsets in a Noetherian space, so this must be a finite intersection. So $W_{s} = W_{t,s}$ for $t$ sufficiently larger than $s$. So the map from $W_t \to W_s$ is a dominant map, for all $t,s$ \st $t>s$.

Now we can get a sequence of spaces $\cdots \to W_{s+1} \to W_{s} \to \cdots$ 
that are subsets of $\cdots,P_{s+1},P_{s},\ldots$ and are non-empty and the maps are dominant.

This implies that $W_s - X_{r,s} \subset Im(W_r \to W_s)$, where $X_{r,s}$ are proper closed subsets. Since $\K$ is uncountable, $W_s - \bigcup X_{r,s} \neq \emptyset$.

Now we can choose a sequence of elements $\phi_s \in W_{s}$ \st $\phi_{s+1} \mapsto \phi_s$.
This gives an element $\phi = (\dots,\phi_s,\phi_{s+1},\dots)$ of the inverse limit $P$.
\end{proof}

\begin{rmk}
As seen in the previous proof, we can always assume that the maps between $P_i$ are dominant.
\end{rmk}

\begin{lemma}\label{basic lemma}
Let $B = \K[\![x_1, \ldots ,x_n]\!]/I$ and $P^1 = \ula{\lim}P^1_i$, $P^2 = \ula{\lim}P^2_i$, $\ldots$, $P^j = \ula{\lim}P^j_i$, $\ldots$  be an infinite sequence of A$\m$AC classes in $B$. 
\begin{enumerate}
\item A$\m$AC classes are closed under finite union. ($P^1 \cup P^2 =  \ula{\lim} P^1_i \cup P^2_i$)
\item A$\m$AC classes are closed under finite intersections. ($P^1 \cap P^2 =  \ula{\lim} P^1_i \cap P^2_i$)
\item A$\m$AC classes are closed under countable intersections. ($\bigcap_j P^j =  \ula{\lim} \bigcap_{j\leq i} P^j_i$ )
\end{enumerate}
Let $B'$ be another such ring, and $Q = \ula{\lim}Q_i$ be an A$\m$AC in $B'$. Let $f: B \to B'$ be a map of local rings. Then,
\begin{enumerate}
\setcounter{enumi}{3}
\item A$\m$AC classes are closed under images. ($f(P^1) =  \ula{\lim} f(P^1_i)$)
\item A$\m$AC classes are closed under inverse images.($f^{-1}(Q) =  \ula{\lim} f^{-1}(Q_i)$)
\end{enumerate}
\end{lemma}
\begin{proof}
Since inverse limit is also a limit it commutes with all other limits.
We immediately obtain \textbf{(2)},\textbf{(5)} and for \textbf{(3)} we obtain 
\[
\bigcap_j P^j = \ula{\lim} \bigcap P^j_i = \ula{\lim} \bigcap_{j\leq i} P^j_i \text{ (because the intersection is countable.)}
\]

We now give a proof of the second equality above. $\ula{\lim} \bigcap P^j_i \subset \ula{\lim} \bigcap_{j\leq i} P^j_i$ as $\bigcap P^j_i \subset \bigcap_{j\leq i} P^j_i$. If $(\ldots,p_i,\ldots,p_2,p_1) \in \bigcap_{j\leq i} P^j_i$, then $\forall j,\ p_j \in P^j_j$ but image of $p_j = p_i$, so $\forall j, \ p_i \in P^j_i$, so $\ p_i \in \bigcap P^j_i$.
\\

%\begin{center}
%\begin{tikzcd}
%\ddots & \arrow[rrrrdddd, -,shift right=5] & & & \vdots & \\
%\ar{r} & P^4_4 \ar{r} & P^4_3 \ar{r} & P^4_2 \ar{r} & P^4_1 &\\
%\ar{r} & P^3_4 \ar{r} & P^3_3 \ar{r} & P^3_2 \ar{r} & P^3_1 &\\
%\ar{r} & P^2_4 \ar{r} & P^2_3 \ar{r} & P^2_2 \ar{r} & P^2_1 &\\
%\cdots \ar{r} & P^1_4 \ar{r} & P^1_3 \ar{r} & P^1_2 \ar{r} & P^1_1 & \ \\
%\end{tikzcd}
%\end{center}

%One can see that the intersection followed by the inverse limit is equal whether we take all things in the diagram or if we only take things under the diagonal.

\textbf{(1)} Follows because inverse limit commutes with finite union

\textbf{(4)} 
Clearly, $f(P^1) \subset  \ula{\lim} f(P^1_i)$. 
\\

Now let $\phi = (\dots,f(p_i),f(p_{i+1}),\dots) \in \ula{\lim} f(P^1_i)$. Take $P^{\phi}_i = f^{-1}(f(p_i)) \cap P^1_i$. Now $P^{\phi}_i$ also forms an inverse system of constructible sets. So $P^{\phi} = \ula{\lim} P^{\phi}_i$ is an A$\m$AC, and by Lemma \ref{nonempty lemma} it is nonempty. Also, since $P^{\phi}_i \subset P^1_i$, we have $P^{\phi} \subset P^1$. For any $\psi \in P^{\phi}$, we have $f(\psi) = \phi$. Therefore $\phi \in f(P^1)$.
\end{proof}

\begin{exmp}\label{basic example}
These are some examples of A$\m$AC classes.
\begin{enumerate}
\item $\{0\} = \ula{\lim}\ \m^{k}/\m^{k}$. (i.e. singletons)

\item equations defining hyperplanes i.e. $\m-\m^2 = \ula{\lim}\ (\m/\m^{k} - \m^{2}/\m^{k})$.

\item $\m^{k} - \m^{l} = \ula{\lim}\ (\m^{k}/\m^{t} - \m^{l}/\m^{t})$.

\item $\{h \in A = \K[\![x_1,\ldots,x_n]\!]\ |\ \tau(h) = r\}$, where $\tau(h) = \len(\O_{\S_{H}})$ and is called Tjurina number. This subset is also $\m$-adically open. (this is a special case of example 6) 

\item $\{h \in A = \K[\![x_1,\ldots,x_n]\!]\ |\ \mu(h) = r\}$, where $\mu(h) = \len(A/\<\frac{\doe h}{\doe x_i}\>)$ and is called the Milnor number of the singularity. This subset is also $\m$-adically open. (the proof similar to Lemma \ref{Tjurina AmAC})

\textbf{Notation:} From now onwards $B = \K[\![x_1,\ldots,x_n]\!]/I$ will denote a ring all whose components have dimension $\a+1$, and $H$ will denote $\Spec(B/\<h\>)$ for any $h \in B$.

\begin{rmk}\label{equi remark}
As $B$ is chosen to be equidimensional of dimension $\a+1$, Let $ZD(B)$ denote the Zero divisor of $B$. For any $h \not\in ZD(B)$, $H$ is equidimensional of dimension $\a$, so $\S^{\a}_{H} = \S_H$.
\end{rmk}

\item $T(r) := \{h \in B\ |\ \len(\O_{\S_{H}}) = r\}$. This subset is also $\m$-adically open. (for a proof see \ref{Tjurina AmAC})

\item $D(d) := \{h \in B \ |\ \dim(\S^{\a}_{H}) \geq d\}$. For the definition of $\S^{\a}_{H}$ see \ref{singular Definition}. (Also see remark \ref{equi remark}.) (for proof see \ref{dimension AmAC})

\begin{rmk}
Note that if $h \in ZD(B)$, then $H$ has dimension $\a+1$. So the $\a \cross \a$ minors will vanish along the $(\a+1)$-dimensional component of $H$. In particular, $\S^\a_H$ will contain this component. Therefore, $ZD(B) \subset D(d) \ \forall d\leq \a+1$. Similarly if $h \not\in ZD(B)$ then $H$ has dimension $\a$ as $\S^\a_H \subset H$, we see that $h \not\in D(\a+1)$. So $D(\a+1) = ZD(B)$ (if $\Spec(B)$ is equidimensional). Finally, note that $D(d) = \emptyset$ if $d > \a+1$.

Though $D(d)$ is an A$\m$AC class, it does not satisfy the hypothesis of Theorem \ref{AmAC result} as $ZD(X) \subset D(d)$ for all $d\leq \a+1$. Nonetheless, the fact that D(d) is an A$\m$AC class is used in the proof of Theorem \ref{AmAC result} makes this an interesting example.
\end{rmk}

\end{enumerate}
\end{exmp}

\begin{lemma}\label{Tjurina AmAC}
$T(r)$ (defined in \ref{basic example}.5) is both A$\m$AC and A$m$AO
\end{lemma}
\begin{proof}
We look at $B_{r+2} = B/\m^{r+2}$ regarded as an affine space. Define $\pi: B \to B_{r+2}$ to be the canonical quotient map.

Using Lemma \ref{singular locus semicontinuity}, $\-T := \{\-h \in  B/\m^{r+2} \ | \ \len(\O_{\S_{H}}/\m^{r+1}) = r\}$ is a Zariski constructible subset. As seen in Lemma \ref{singular locus semicontinuity}, $\len(\O_{\S_{H}}/\m^{r+1})$ does not depend on which lift $h$ of $\-h$ is chosen.

We will prove that $T(r) = \pi^{-1}(\-T)$. 

First, let $h \in T(r)$. Since $\len(\O_{\S_{H}}) = r$, $\m^{r+1}\O_{\S_{H}} = 0$, so  $\O_{\S_{H}} = \O_{\S_{H}}/\m^{r+1}$, so $\pi(h)\in\-T$
Let $h \in \pi^{-1}(\-T)$ and let $\-h := \pi(h)$. If $\dim(\O_{\S_{H}}) > 0$, then by Lemma \ref{dim bound lemma}(1), $\len(\O_{\S_{H}}/\m^{r+1}) \geq r+1$, which is a contradiction to the choice of $h$. 
So $\dim(\O_{\S_{H}}) = 0$ i.e. $h$ has an isolated singularity.

Let multiplicity of $\S_H$ = $\len(\O_{\S_{H}}) \geq r+1$, then by Lemma \ref{dim bound lemma}(2) $\len(\O_{\S_{H}}/\m^{r+1}) \geq r+1$, which is a contradiction to the choice of $h$. So $\len(\O_{\S_{H}}) \leq r$, which implies that $\O_{\S_{H}}/\m^{r+1} = \O_{\S_{H}}$. 
So $\len(\O_{\S_{H}}) = r$, This implies $h \in T(r)$.
\end{proof}

\begin{lemma}\label{dimension AmAC}
$D(d)$ (defined in \ref{basic example}.7) is a an A$\m$AC.
\end{lemma}
\begin{proof}
We will prove that $D(d) = \ula{\lim} D_i(d)$, where $D_i(d)$ is defined as follows:

$$D_r(d) := \{\-h \in B/\m^{r} \ |\ \len(\O_{\S^{\a}_{H}}/\m^{k+1}) \geq {k+d \choose d}, \ \forall \ k \leq r-2\}$$
Using Lemma \ref{singular locus semicontinuity}, $D_r(d)$ is a Zariski constructible subset.

$D(d) \subset \ula{\lim} D_i(d)$ because if $\dim(\S^{\a}_{H}) \geq r$ then $\-h = h+\m^{r}$ satisfies the condition of $D_r(d)$ by Lemma \ref{dim bound lemma}(1).

Let $h \in B$ and $h \not\in D(d)$, then $\len(\O_{\S^{\a}_{H}}/\m^{k+1})$ is a polynomial of degree $< d$ (Lemma \ref{dimension}). So for large $k$, $\len(\O_{\S^{\a}_{H}}/\m^{k+1})<{k+d \choose d}$ which is a higher degree polynomial. So $h \not\in \ula{\lim} D_i(d)$.
\end{proof}

Gurjar's Theorem in \cite{GurjarHST} can be understood as ``For a sequence of hyperplanes $H_i$ if $X \cap H_i$ has a 0-dimensional singular locus and Milnor number tends to $\infty$ as $i \to \infty$, then there exists a hyperplane $H$ \st dimension of singular locus of $X \cap H$ is $\geq 1$". 
We want to generalize to the following case: ``For every integer $i \geq 0$, if the dimension of singular locus of $X \cap H_i$ is $d$ and some\footnote{The invariant we use is the multiplicity of singular locus. There can possibly be other invariants which can lead to similar results.} invariant tends to infinity then there exists a hyperplane $H$ \st dimension of singular locus of $X \cap H$ is $\geq d+1$".
But why should we stop at a jump of $d$ to $d+1$? Why not give a theorem statement for a $d$ to $d+r$ jump? 
\\

Note that, as seen in Section 2, understanding the growth of the Hilbert-Samuel polynomial determines the dimension of variety. This brings us to the following Lemma with a technical growth condition. Note that the Lemma puts no constraint on the dimension of $\S_{X \cap H_i}$.

\begin{lemma}\label{AmAC lemma}
Let $X$ be an equidimensional closed subvariety of $\Spec(A)$. Let $ZD(X)$ denote the zero divisor of $\O_X$ in $A$. Let $P \subset A$ be an algebraic $\m$-adically closed class \st $ZD(X) \cap P = \emptyset$. Let $h_i \in P$ be an infinite sequence of hypersurfaces \st $\len(\O_{\S_{X \cap H_i}}/\m^{k+1}) \geq {k + d+r \choose d+r}$ for all $k<i$. Then there exists a hypersurface $h \in P$ \st $\S_{X \cap H}$ is of dimension $\geq d+r$.
\end{lemma}

\begin{proof}
Let $X = \Spec(B)$. Define $\a + 1 := \dim{X}$ so if $h \not\in ZD(X)$ then $\S^{\a}_{H} = \S_{H}$ (where $H$ denotes $\Spec(B/\<h\>)$) (see remark \ref{equi remark}).

Let $D(d+r) = \{h \in B \ |\ \dim(\S^{\a}_{H}) \geq d+r\}$ (where $H$ denotes $\Spec(B/\<h\>)$)
As seen in example \ref{basic example}.7, $D(d+r)=\ula{\lim}D_i(d+r)$ where $$D_i(d+r) = \{\-h \in B/\m^{i} \ |\ \len(\O_{\S^{\a}_{H}}/\m^{k+1}) \geq {k+d+r \choose d+r}, \ \forall \ k \leq i-1\}$$
Let map $f: A \to B$ be the map induced by the inclusion of $X \to \Spec(A)$. 
Let $P = \ula{\lim} P_i$.

The hypothesis implies that for every $i>0$, there exists a hypersurface $h_i \in P$ \st $\len(\S_{H_i\cap X}/\m^{k+1}) \geq {k+d+r \choose d+r}$, forall $k \leq i-1$.
By Lemma \ref{dimension AmAC}, we know that $D(d+r) = \ula{\lim}D_i(d+r)$ is an A$\m$AC. 
Applying Lemma \ref{basic lemma}, we get $P \cap f^{-1}(D(d+r)) = \ula{\lim} (P_i \cap f^{-1}(D_i(d+r)))$ is an A$\m$AC. 
Now using Lemma \ref{nonempty lemma}, gives us that $P \cap f^{-1}(D(d+r)) \neq \emptyset$.
\end{proof}

\begin{thm}\label{AmAC result}
\AmACResult
\end{thm}

\begin{proof}
Because of the hypothesis, there exists  a hyperplane $h_i \in P$ \st $\dim(\S_{H_i\cap X}) = d$ and the multiplicity of $\S_{H_i\cap X}$ is greater than or equal to $i$. 
Lemma \ref{dim bound lemma}(2) implies that $\len(\O_{\S_{H_i\cap X}/\m^{k+1}}) \geq {k+d+1 \choose d+1}$ $\forall k<i$.
So $h_i$ satisfies the hypothesis of Lemma \ref{AmAC lemma} (with $r = 1$).
\end{proof}

\begin{cor}\label{Hyperplane section theorem}
\HyperplaneSectionTheorem
\end{cor}
\begin{proof}
Zero divisors of $X$ correspond to component of $X$, But since $X$ is irreducible and not contained in any hyperplane, $ZD(X) \cap \m - \m^2 = \emptyset$.
We can now apply Theorem \ref{AmAC result} to complete the proof.
\end{proof}

\begin{thm}\label{multi result}
Let $X$ be an irreducible closed subvariety of $\Spec(A)$ and $P \subset A$ be an algebraic $\m$-adically closed class. Let $h_i \in P$ be an infinite sequence of hypersurfaces \st the multiplicity of $X \cap H_i$ tends to $\inf$ as $i \to \inf$. Then there exists a hypersurface $h \in P$ \st $X \subset H$.
\end{thm}

\begin{proof}
The proof is similar to the proof of the previous Theorem.

Let $X = \Spec(B)$.
Let map $f: A \to B$ be the map induced by the inclusion of $X \to \Spec(A)$. 
Let $P = \ula{\lim} P_i$. Note that.
\[\{0\} = ZD(B) = \{h \in B \ | \ \dim(H) \geq \a+1\} = \ula{\lim} Z_i \text{ (proof is similar to \ref{dimension AmAC})}\]
\[\text{where, } Z_i = \{\-h \in B/\m^{i} \ |\ \len(\O_{H}/\m^{k+1}) \geq {k+\a+1 \choose \a+1}, \ \forall \ k \leq i-1\}\]
The hypothesis implies that for every $i>0$, there exists a hypersurface $h_i \in P$ \st $\len(\O_{H_i}/m^{k+1}) \geq {k+\a+1 \choose \a+1}$, forall $k \leq i-1$.
Applying Lemma \ref{basic lemma}, we get $P \cap f^{-1}(ZD(B)) = \ula{\lim} (P_i \cap f^{-1}(Z_i))$ is an A$\m$AC. 
Now using Lemma \ref{nonempty lemma}, gives us that $P \cap f^{-1}(ZD(B)) \neq \emptyset$.
Since $X$ is reduced and irreducible, $ZD(B) = \{0\}$. 
So $P \cap f^{-1}(0)$ is non-empty.
\end{proof}

We would like to end this Section by proving the following result using A$\m$AC.

\begin{thm}\label{milnor tjurina}
\milnorTjurina
\end{thm}
\begin{proof}
Let $\pi_r : A \to A/\m^r$ and $\pi_{s,r}: A/\m^s \to A/\m^r$ denote the canonical projection maps.

\textbf{Claim:} $D(1) = \{h \in A = \K[\![x_1,\ldots,x_n]\!]\ |\ \mu(h) = \infty\}$ is A$\m$AC. 

Note the set given by $K(\neq r) = \{h \in A \ |\ \mu(h) \neq r\} = \pi_{r+1}^{-1}(\{\-h \in A/\m^{r+1} \ |\ \mu(\-h) \neq r\})$ is A$\m$AC (as complement is $\m$-adically open). (Similar to Lemma \ref{Tjurina AmAC}).
Similarly $K(>r) = \{h \in A \ |\ \mu(h) > r\} = \pi_{r+1}^{-1}(\{\-h \in A/\m^{r+1} \ |\ \mu(\-h) > r\})$ is also an A$\m$AC.
$$D(1) = \bigcap_r K(\neq r) = \bigcap_r K(>r) = \ula{\lim} \-K(>r)$$
where $\-K(>r) := \{\-h \in A/\m^{r+1} \ |\ \mu(\-h) > r\}$.
Note $D(1) = \{h \in A \ |\ \dim(\S_H) \geq 1\}$.

Let $T(\tau) = \{h \in A \ |\ \tau(h) = r\}$, as defined in example \ref{basic example}.4. 
As seen in Lemma \ref{Tjurina AmAC}, $T(\tau) = \pi_{r+1}^{-1}(\-T(\tau))$ is an A$\m$AC.
\vspace{10pt}

Note $T(\tau)$ is a collection of all possible hypersurfaces which satisfy the hypothesis of the Theorem (albeit with repetition up to isomorphism).

Assume for a contradiction $\mu$ is not bounded on the set $T(\tau)$.
\begin{align} \nonumber
\begin{split}
& T(\tau) \cap K(>r) \neq \emptyset \quad \forall r \text{ sufficiently larger than } \tau\\
\implies &\pi_{r,\tau}^{-1}(\-T(\tau)) \cap \-K(>r) \neq \emptyset \quad \forall r \text{ sufficiently larger than } \tau \quad \text{ (by nonemptiness Lemma \ref{nonempty lemma})}\\
\implies &T(\tau) \cap \ula{\lim} \-K(>r) \neq \emptyset \ \ \quad \text{ (by nonemptiness Lemma \ref{nonempty lemma})}\\
\implies &\-T(\tau) \cap D(1) \neq \emptyset
\end{split}
\end{align}
This contradicts the fact that any $h \in \-T(\tau)$ can only have isolated singularity.
\end{proof}

%\begin{rmk}
%one notes that the proof above seems to work for any two invariant \st
%\begin{enumerate}
%\item they are finitely determined i.e $\forall \ \e \ \exists \ r(\e)$ with the property that the invariant is less $\e$ can be determined modulo $\m^{r(\e)}$.
%\item they are upper semi-continuous (to say that they define A$\m$AC)
%\item they are finite iff hypersurface has an isolated singularity
%\end{enumerate}
%This in particular also gives us the $\tau$ is bounded if $\mu$ is constant. But we already know that as $\tau \leq \mu$.
%\end{rmk}

\begin{rmk}
The proof also shows us that, the same A$\m$AC class can be obtained as the inverse limit of two different inverse systems. (For example, the inverse system used to define $D(1)$ in example \ref{basic example}.7  is different from the inverse system used in this proof)
\end{rmk}
\section{Another result for singular surfaces}

Before we get to the result, we will need a refinement of Lemma \ref{dim bound lemma} in the case of rings of dimension 1. We also provide and make use of more examples of A$\m$AC.

\begin{defn}
Let $R$ be a Noetherian ring and $I$ be an ideal in $R$. We define the \textbf{torsion of $R$ supported on $V(I)$} as $$T(R,I) = \{x \in R| I^n x = 0 \text{ for some }n\} = \ura{lim} \ Hom(\frac{R}{(I^n)},R)$$

Note that since $R$ is Noetherian, $\sqrt{I}^n \subset I \subset \sqrt{I}$, so $T(R,I) = T(R,\sqrt{I})$. In particular, $T(R,I)$ only depends on $V(I)$, the variety defined by $I$.
\end{defn}

\begin{lemma}\label{dim 1 torsion bound lemma}
Let $K$ be an infinite field. Let $R$ be a Noetherian local complete K-algebra of dimension $1$ with maximal ideal $\m$ \st $R/\m = K$. Let  $t$ and $e$ be $\len(T(R,\m))$ and $e(R,\m)$ respectively. Then we have the following.

For a generic $x \in \m-\m^2$, and for all $k \in \bb{Z}_{\geq 0}$
\begin{equation}ek \leq \len(\frac{R}{\<x^{k+1}\>}) \leq ek + t\end{equation}
And if $k\leq t$, we can improve the lower bound as follows.
\begin{equation}\len(\frac{R}{\<x^{k+1}\>}) \geq (e+1)k \end{equation}
And if $k \geq t$, the upper bound is in fact achieved.
\begin{equation}\len(\frac{R}{\<x^{k+1}\>}) = ek + t \end{equation}
\end{lemma}
\begin{proof}
Use Noether normalization to get that $R$ is integral over $K[\![x]\!]$.
Also, choose $x$ so that the multiplicity of $R$ w.r.t $\<x\>$ is the same as the multiplicity of $R$ w.r.t $\m$.

But $K[\![x]\!]$ is a PID. So torsion-free finitely generated modules are free.
Torsion w.r.t $x$ is the same as the torsion w.r.t $\m$ as $\sqrt{\<x\>} = \m$.
So $R/T(R,\m)$ is a free $K[\![x]\!]$ module. Let $\a$ be its rank. Now,
\begin{align} \nonumber
\begin{split}
e & = \text{the multiplicity of } R \text{ w.r.t } \<x\>\\
& = \text{the multiplicity of } R/T(R,\m) \text{ w.r.t } \<x\>\\
& = \text{the rank of } R/T(R,\m) \text{ over } K[\![x]\!] \ (\text{As the Hilbert polynomial of the free module is } \a n)\\
& = \a 
\end{split}
\end{align}

In particular $R \iso K[\![x]\!]^{e} \oplus T(R,\m)$ as a $K[\![x]\!]$-module.
The inequalities (4) are just the length of this module (modulo $\<x^k\>$) with and without torsion. 
Also note that for large $k$, the length must be $ek+t$.
\\

For (5) and (6), we first note that
\[\len\bigg(\frac{R}{\<x^{k+1}\>}\bigg) = \len\bigg(\frac{R}{\<x\>}\oplus\frac{\<x\>}{\<x^{2}\>}\oplus \cdots \oplus \frac{\<x^{k}\>}{\<x^{k+1}\>}\bigg) = \sum_{i=0}^k \len\bigg(\frac{\<x^{i}\>}{\<x^{i+1}\>}\bigg)\]
We have a sequence of maps 
\[\frac{R}{\<x\>} \xra{\cross x} \frac{\<x\>}{\<x^{2}\>} \xra{\cross x} \cdots \xra{\cross x} \frac{\<x^{k}\>}{\<x^{k+1}\>} \xra{\cross x} \frac{\<x^{k+1}\>}{\<x^{k+2}\>} \xra{\cross x} \cdots\]
We see that each of these maps is surjective. In particular, the length of these is a decreasing function and eventually must become a constant $e$ (which is the rank of the free module).
So there exists $q$ \st
\[\len(\frac{\<x^{i}\>}{\<x^{i+1}\>}) 
\begin{cases}
\geq e+1, \text{ if } i\leq q\\
= e, \text{ if } i > q
\end{cases}\]
Now $q$ can be at most $t$ because we have a finite amount of torsion given by $\len(T(R,\m))$.

Finally, we sum up the inequalities above, from $i = 1$ to $k$. Since we are adding $e$, for every $k>q$, and $ek+t$ must be eventually obtained. So $ek+t$ must be obtained at $q$ itself. In conclusion, after summing up we get the following which proves the Lemma:
\[\len(\frac{R}{\<x^{k+1}\>})
\begin{cases}
\geq (e+1)k, \text{ if } k\leq q\\
= ek + t \geq (e+1)k, \text{ if } q < k\leq t\\
= ek + t, \text{ if } t < k
\end{cases}\]

\end{proof}

\begin{exmp}\label{advanced example}
Here are some more examples of A$\m$AC. As before let $B = \K[\![x_1,\ldots,x_2]\!]/I$ be \st $B$ is equidimensional of  dimension $\a +1$ and we denote $H = \Spec(B/\<h\>)$, for any $h \in B$.

\begin{enumerate}
\setcounter{enumi}{7}
\item $M(d,e) := \{h \in B\ |\ \dim(\S^{\a}_{H}) \geq d \text{ and if } \dim(\S^{\a}_{H})=d \text{ then } m(\S^{\a}_{H}) \geq e\}$ (where $m(\S^{\a}_{H})$ is the multiplicity of $\S^{\a}_{H}$ w.r.t. $\m$) (for proof see lemma \ref{multiplicity AmAC})

\begin{rmk}
We note that $M(d,e) \subset M(d',e')$ if $(d , e) \geq (d' , e')$ in the lexicographic ordering. This gives us a stratification of the ideal $\m$. We can obtain a stratification of $P \subset \m$, simply by intersecting with this stratification. Note that the smallest $(d,e)$ (in the lexicographic ordering) \st $M(d,e) \cap P$ is non-empty  $M(d,e)$ has the property that $M(d,e) \cap P$ is $\m$-adically open in $P$.
\end{rmk}

\item $C(e,t) := \{h \in B\ |\ \dim(\S_{H}) = 1 \text{ and } m(\S_{H}) = e \text{ and } \len(T (\O_{\S_{H}}, m)) = t\}$ (where $m(\S_{H})$ is the multiplicity of $\S_{H}$ w.r.t to $\m$) (for proof see lemma \ref{curve AmAC})

\end{enumerate}
\end{exmp}

\begin{lemma}\label{multiplicity AmAC}
$M(d,e)$ (defined in example \ref{advanced example}.8) is an A$\m$AC.
\end{lemma}

\begin{proof}
\begin{align} \nonumber
\begin{split}
M_r := \{\-h \in B/\m^{r} \ |\ \len(\O_{\S^{\a}_{H}}/\m^{k+1}) &\geq \sum_{i=0}^{\min(e,d)} (-1)^{i} {e \choose i+1}{n+d-i\choose d-i},\\ &\forall \ k \text{ \st } e-1 \leq k \leq r-2\}
\end{split}
\end{align}
As follows from Lemma \ref{singular locus semicontinuity}, $M_r$ is a Zariski constructible subset.
We will prove that $M(d,e) = \ula{\lim} M_r$. The proof is very similar to the proof of example \ref{basic example}.7.

$M(d,e) \subset \ula{\lim} M_r$ because if $h \in M(d,e)$ then $\-h = h+\m^{r}$ satisfies the condition of $M_r$ by Lemma \ref{dim bound lemma}(3).

Let $h \in B$ and $h \not\in M(d,e)$. Then, $\len(\O_{\S^{\a}_{H}}/\m^{k+1})$ is a polynomial of degree $\leq d$ and if its degree is $d$ then its highest coefficient is lesser than $e/d!$ (Lemma \ref{dimension} and definition of multiplicity). 
So for large $k$, $\len(\O_{\S^{\a}_{H}}/\m^{k+1}) < \sum_{i=0}^{\min(e,d)} (-1)^{i} {e \choose i+1}{n+d-i\choose d-i}$ which either has a higher degree or a higher leading coefficient. 
So $\-h \not\in M_k$, and thus $h \not\in \ula{\lim} M_i$.
This proves that $\ula{\lim} M_r \subset M(d,e)$.
\end{proof}

\begin{lemma}\label{curve AmAC}
$C(e,t)$ (defined in \ref{advanced example}.9) is an A$\m$AC.
\end{lemma}
\begin{proof}
For the sake of brevity, we define 
$$S_h(x,k,r):=\len(\frac{\O_{\S_{H}}}{\<x^k\>+\m^{r}})$$
\begin{align} \nonumber
\begin{split}
C_r := \{\-h \in B/\m^{r} \ |\ & \forall \ k > t \text{ \st } ke + t \leq r-2\\
& \min_{x \in \m-\m^2}(S_h(x,k,r-1)) = ke + t \}
\end{split}
\end{align}
First, we will prove that $C_r$ is a Zariski constructible set.
For this we define
\begin{align} \nonumber
\begin{split}
\Pi_r(\a) := \{(\-h,x) \in B/\m^{r} \cross (\m/\m^r-\m^2/\m^r) \ | 
& \ S_h(x,k,r-1) = \a, \\ 
&\forall \ k > t \text{ \st } ke + t \leq r-2\}
\end{split}
\end{align}
\begin{align} \nonumber
\begin{split}
P_r(\a) := \{ \-h \in B/\m^{r} \ | \exists \ x \in (\m/\m^r-\m^2/\m^r)\text{ \st }
& \ S_h(x,k,r-1) = \a, \\ 
&\forall \ k > t \text{ \st } ke + t \leq r-2\}
\end{split}
\end{align}

Now $\Pi_r(\a)$ is a Zariski constructible set, as it is the inverse image of an upper semi-continuous function (produced using arguments of Lemma \ref{basic semicontinuity} and \ref{singular locus semicontinuity}). And so its first projection $P_r(\a)$ is also Zariski constructible.
Finally $C_r = P_r(ke+t) - \bigcup_{i=0}^{ke+t-1}P_r(i)$ is also Zariski constructible.

We will prove that $C(e,t) = \ula{\lim} C_i$.

$C(e,t) \subset \ula{\lim} C_i$ because if $h \in C(e,t)$, then $\-h = h+\m^{r}$ satisfies the condition of $C_r$ by Lemma \ref{dim 1 torsion bound lemma}(6).

Let $h \in  \ula{\lim} C_i$, then $S_h(x,k,r-1) = \len(\frac{\O_{\S_{H}}}{\<x^k\>+\m^{r-1}}) = ke + t$. But since $r-1 > ke + t$, $\O_{\S_{H}}/\<x^k\> = \O_{\S_{H}}/(\<x^k\>
+\m^{r-1})$. This implies that $\len(\O_{\S_{H}}/\<x^k\>) = ke + t$, which is a polynomial of degree $1$, the highest coefficient of this polynomial is $e$ and the constant term is $t$.
This implies $\dim \S_{H} = 1$ (Lemma \ref{dimension}).
Note that the set of all $x$ \st $x$ gives minimum length, is a Zariski open set modulo $\m^r$.
Therefore if $x$ has the minimum length then $x$ can be chosen to be a general element. So Lemma \ref{dim 1 torsion bound lemma}(6) can be applied to conclude that $\S_{H}$ has multiplicity $= e$ and torsion $= t$. Therefore $h \in C(e,t)$. 
This proves $\ula{\lim} C_i \subset C(e,t)$.
\end{proof}

In the next Lemma, we obtain $M(1,e)$ from example \ref{advanced example}.8, as an inverse limit of different sets.

\begin{lemma}
$$M(1,e) = D(2) \cup \ula{\lim} Q_r$$
\begin{align} \nonumber
\begin{split}
\text{where} \ Q_r := \{\-h \in B/\m^{r} \ |\ \forall \ x \in \m-\m^2 \text{ \st } &\len(\frac{\O_{\S^{\a}_{H}}}{\<x^k\>+\m^{r}}) \geq ek, \\ &\forall \ k \text{ \st } ek \leq r-2\}
\end{split}
\end{align}
\end{lemma}

\begin{proof}

As seen before $Q_r$ is a Zariski constructible subset. Suppose $h \in M(1,e)$, then either $\dim(\S^{\a}_h) \geq 2$, in which case $h \in D(2)$, or $\dim(\S^{\a}_h) = 1$. Then, by Lemma \ref{dim 1 torsion bound lemma}(4), $\-h \in Q_r$. So $M(1,e) \subset D(2) \cup \ula{\lim} Q_r$.

Clearly, $D(2) \subset M(1,e)$.

Let $h \in \ula{\lim} Q_r$, then $\len(\frac{\O_{\S^{\a}_{H}}}{\<x^k\>+\m^{r}})$ is unbounded, so $\len(\O_{\S^{\a}_{H}})$ is infinite, which implies $\dim(\S^{\a}_{H}) > 0$.  If $\dim(\S^{\a}_H)\geq 2$, then $h \in D(2) \subset M(1,e)$. If $\dim(\S^{\a}_H) = 1$, then we can apply Lemma \ref{dim 1 torsion bound lemma}(4) to conclude that multiplicity of $\S^{\a}_{H}$ is $\geq e$. So $h \in M(1,e)$.
Therefore, $M(1,e) = D(2) \cup \ula{\lim}Q_r$.
\end{proof}

Observe that in Theorems \ref{AmAC result} and \ref{multi result}, as multiplicity tends to infinity we get a variety of higher dimension, but the multiplicity of this new variety is not infinity anymore. This is because the multiplicity of different dimensional variety are defined as coefficients of different degree terms of the Hilbert polynomial. 
This leads us to believe that given any Noetherian ring $R$, there is a multiplicity associated with every natural number $i$ (i.e. $i \in \bb{Z}_{\geq 0}$)) representing the multiplicity of $R$ supported on $i$ dimensional primes. So if $\dim(R) \leq d$ then the multiplicity associated with all $i>d$ vanishes.
\\

For example, if $\dim(R)\leq 1$ then we get the following two numbers: (i) the length of torsion supported on maximal ideals, and (ii) the multiplicity of $R$ supported on dimension 1 primes.
As shown after Definition \ref{multiplicity tuple}, the multiplicity supported on dimension 1 primes is the multiplicity of $R$ if $\dim(R)=1$ and $0$ if $\dim(R)=0$. The multiplicities of dimension zero rings are supported on dimension $0$ primes (i.e. torsion points), which we have already accounted for. This leads us to the following definition:

\begin{defn}\label{multiplicity tuple}
Let $K$ be an infinite field.
Let $R$ be a local $K$-algebra of dimension $\leq 1$ with maximal ideal $\m$ \st $R/\m = K$.
Use Noether normalization to get that $R$ is integral over $A = K[\![x]\!]$.
Also, choose $x$ so that the multiplicity of $R$ w.r.t $\<x\>$ is the same as the multiplicity of $R$ w.r.t $\m$. 
Let $\chi^{\<x\>}(k) = ek+t$ for a large enough $k$. Define the multiplicity tuple of $R$ to be $(e,t)$.
\end{defn}

To illustrate the discussion above we will compute the multiplicity tuple in case of dimension 1 and dimension 0 rings.
If $\dim (R) = 1$ and the multiplicity of $R$ is $e$ and $\len(T(R,m)) = t$, then
in light of Lemma \ref{dim 1 torsion bound lemma}(6), the multiplicity tuple is $(e,t)$
If $\dim (R) = 0$ then the multiplicity of $R$ is $\len(R) = \len(T(R,m)) = t$. we can conclude that the multiplicity tuple is $(0,t)$
\\

Now as the multiplicity of zero-dimensional rings tends to infinity, we get a ring of dimension $1$. This suggests that $(1,0) > (0,n) \forall n$, which indicates we should use some kind of lexicographic ordering. This is demonstrated in the following Theorem:

\begin{thm}\label{surface result}
\surfaceResult
\end{thm}
\begin{proof}

Since we have assumed that $\dim B = 2$, any hypersurface defined by $h \not\in ZD(X)$  must have dimension 1. So if $h \in P$, then $\dim(H) = 1$, which implies $\dim(\S_{H}) \leq 1$.
Now assume for the sake of  contradiction that their multiplicity is unbounded, then by Theorem \ref{AmAC result}, there is a hypersurface with singular locus of dimension $> 1$, which is a contradiction. So the multiplicity is bounded by some number $e$.
We chose $e = e(P)$ to be the maximum achieved by any of the hypersurfaces. In the case that all hypersurfaces (in $P$) intersect in isolated singularities, then $e(P) = 0$.

Now if $h$ is a hypersurface \st the multiplicity of $\S_{H}$ is less than $e$, then the multiplicity tuple is already less in the lexicographic ordering. So we define $L$ as follows.

\begin{align} \nonumber
\begin{split}
L := &\{h \in P \ |\ \dim(\S_{H}) = 1 \text{ and } m(\S_{H}) = e(P)\} \\= &\{h \in P \ |\ \dim(\S_{H}) \geq 1 \text{ and if } \dim(\S_{H}) =1 \text{ then } m(\S_{H}) \geq e(P)\} = M(1,e(P)) \cap P
\end{split}
\end{align}

Note that $M(d,e)$ has been defined in Example \ref{advanced example}.8. So $L$ is an A$\m$AC (by lemma \ref{basic lemma}).

Assume for the sake of contradiction that the $\m$-torsion of $\S_{H}$ is unbounded for $h \in L$.

As seen in Lemma \ref{advanced example}, $M(1,e(P)+1) = D(2) \cup \ula{lim} Q_r$, where $D(2) = ZD(B)$ and $Q_r$ is as follows.
\begin{align} \nonumber
\begin{split}
Q_r = \{h \in B/\m^{r} \ |\ \forall \ x \in \m-\m^2 \text{ \st } &\len\bigg(\frac{\O_{\S^{\a}_{H}}}{\<x^k\>+\m^{r}}\bigg) \geq (e(P)+1)k, \\ &\forall \ k \text{ \st } (e(P)+1)k \leq r-2\}
\end{split}
\end{align}

Since $\m$-torsion is unbounded, there is a sequence of $h_i \in L$, \st $\len(T(\O_{\S_{H}},\m)) \geq i$. But by Lemma \ref{dim 1 torsion bound lemma}(5), $\-h_i \in Q_i$. In particular, $Q_i\cap L_i$ is nonempty for each $i$. Now by Lemma \ref{nonempty lemma}, $M(1,e(P)+1) \cap L$ is nonempty, but $M(1,e(P)+1) \cap L$ is empty since $e(P)$ is the maximum possible multiplicity achieved. This is a contradiction to the assumption.

So the $\m$-torsion is bounded (in $L$). Now choose a hypersurface $h \in L$ which has the maximum $\m$-torsion $t(P)$. Then $\S_{H}$ achieves the maximum multiplicity tuple $(e(P),t(P))$. 
\end{proof}

\begin{exmp}
We will compute the above invariant in the case of quadratic singularity. Fix $A = \bb{C}[\![x,y,z]\!]$ , $X = \Spec(B) = \Spec(A/\<x^2+y^2+z^2\>)$ and $P = \pi(\m_A - \m_A^2) = \m_B - \m_B^2$

Let $H$ be a hyperplane given by $h \in A$. Without loss of generality assume that $\frac{\doe h}{\doe z} \neq 0$. Then by implicit function Theorem assume that $H$ is given by $z = f(x,y)$. 
\[ H = \Spec(\K[\![x,y]\!]) \xra{(x,y) \mapsto (x,y,f(x,y))} \Spec(\K[\![x,y,z]\!])\]
Now we see that, $H \cap X$ is given by $x^2+y^2+f^2(x,y) = 0$ as a subset of $H = \Spec(\K[\![x,y]\!])$.

Now $\S_{H\cap X}$ is of dimension $1$, iff $H \cap X$ is not reduced. 
Assume that $x^2+y^2+f^2(x,y) = g^2q$ \st $g \in \m$ and $q \in A$.
If $q \in \m$, then $g^2q$ will be of order $3$ which is not possible. Therefore $q$ is a unit in the power series ring.
So $q = p^2$, and so we may assume $g = gp$ and $q = 1$.

$$x^2 + y^2 + f^2(x,y) = g^2$$

Again if $g \in \m^2$, then the order of $g^2$ is greater than or equal to 4, which is again not possible. So $g\in \m-\m^2$. Since $H \cap X$ is given by $g^2$, $\S_{H\cap X}$ must be given by $g$. Therefore multiplicity of $\S_{H\cap X}$ is 1 and torsion is $0$.

So we get that any multiplicity tuple of $\S_{H\cap X}$ must be bounded by $(1,0)$.

To see that this bound is achieved, we use the hyperplane given by $z=iy$, The intersection is given by $x^2 = 0$.
\end{exmp}

\section{The analytic case}

Fix $\ring{A}$ be the ring of algebraic power series\footnote{ The ring of algebraic power is defined as the ring of all formal power series algebraic over polynomials, which can also be obtained as the henselization of $\K[x_1,\ldots,x_n]$ at maximal ideal $\<x_1,\ldots,x_n\>$} or convergent power series (if $\K = \bb{C}$) and let $\ring{\m}$ be its maximal ideal. 

\begin{defn}
A subset $P \subset B = \ring{A}/I$ is said to be a \textbf{finitely determined A$\m$AC} or simply \textbf{finitely determined class}, if $P =  \pi_i^{-1} (P_i)$ for some $P_i$ which is a constructible subset of $B/\m^i$ (in Zariski topology).
\end{defn}

\begin{rmk}
Note that finitely determined classes are closed under finite intersection, finite union and complement. So they are also $\m$-adically open.
\end{rmk}

\begin{exmp}
Example \ref{basic example}, (2)-(6) are finitely determined i.e.
hyperplanes, $\m^{k} - \m^{l}$, Tjurina number constant, Milnor number constant etc.
\end{exmp}

\begin{thm}\label{analytic result}
Let $\ring{X}$ be an equidimensional subvariety of $\Spec(\ring{A})$. Let $ZD(\ring{X})$ denote zero-divisor of $\O_{\ring{X}}$ in $\ring{A}$. Let $\ring{P} \subset \ring{A}$ be a finitely determined class of functions \st $ZD(\ring{X}) \cap P = \emptyset$. Let $h_i \in \ring{P}$ be an infinite sequence of hyperplanes \st $\S_{\ring{X} \cap \ring{H}_i}$ is of dimension $d$ and the multiplicity of $\S_{\ring{X} \cap \ring{H}_i} \to \inf$ as $i \to \inf$. Then there exists a hyperplane $h \in \ring{P}$ \st $\S_{\ring{X} \cap \ring{H}}$ is of dimension $\geq d+1$.
\end{thm}
\begin{proof}
The proof of this statement is a slight modification of the argument given in \cite{MoriHST}. For the reader's convenience, we detail it out here. For clarity, during the course of the proof of Theorem \ref{analytic result}, $\^{A}$ will denote the formal power series ring (instead of the usual $A$).

Let $\ring{I} = \<f_1,\dots,f_k\>$ denote the ideal defining $\ring{X}$ and $\^{I} := \ring{I}\^{A}$. Let $\^{X}$ be the subscheme defined by $\^{I}$. By Theorem \ref{AmAC result}, a formal power series $\^{h}$ with the required property exists. We will now use Artin approximation to get the required element in $\ring{A}$.

Since dimension $\S_{\^{X}\cap \^{H}} \geq d+1$, there exist a $d+1$ dimensional irreducible subvariety $\^{S} \subset \S_{\^{X}\cap \^{H}}$. Let $ \^{\mf{p}}$ be the prime ideal defining $\^{S}$.
Now, we split the proof into two cases.

\textbf{Case 1:} $\^{S} \subset \S_{\^{X}}$. Since there is a hypersurface \st $\S_{\ring{X} \cap \ring{H}_i}$ is of dimension $d$, $\dim \S_{\^{X}} \leq d+1$. But $\dim \^{S} = d+1$, so $\^{S}$ is a component of the reduced structure of $\S_{\^{X}}$.
We take the corresponding component of $\S_{\ring{X}}$ and call it $\ring{S}$. In other words, $\^{S}$ already comes from a subscheme of $\Spec(\ring{A})$. We know that $\^{S} \subset \S_{\^{X} \cap \^{H}} \subset \^{H}$. So by Artin approximation, we can find a $\ring{H}$, \st $\ring{S} \subset \ring{H}$.

\textbf{Case 2:} $\^{S} \not\subset \S_{\^{X}}$. Then $\^{X}$ is smooth at the generic point of $\^{S}$. Let us now analyze the dimension of cotangent space of $\Spec(\^{A})$, $\^{X}$, and $\^{X}\cap \^{H}$ at point $\^{\mf{p}}$. Here $k(\^{\mf{p}}) =  \frac{\^{A}_{\^{\mf{p}}}}{\^{\mf{p}}\^{A}_{\^{\mf{p}}}}$
\begin{align} \nonumber
\begin{split}
\dim_{k(\^{\mf{p}})} T^*_{\^{\mf{p}}}(\Spec(\^{A})) &= \frac{\^{\mf{p}}\^{A}_{\^{\mf{p}}}}{\^{\mf{p}}^2\^{A}_{\^{\mf{p}}}} = n-(d+1)\\
\dim_{k(\^{\mf{p}})} T^*_{\^{\mf{p}}}(\^{X}) &= \frac{\^{\mf{p}}\^{A}_{\^{\mf{p}}}}{(\^{\mf{p}}^2 + I)\^{A}_{\^{\mf{p}}}} = (\a +1)-(d+1) \text{ because $\^{X}$ is a smooth at $\^{\mf{p}}$}\\
\dim_{k(\^{\mf{p}})} T^*_{\^{\mf{p}}}(\^{X} \cap \^{H}) &= \frac{\^{\mf{p}}\^{A}_{\^{\mf{p}}}}{(\^{\mf{p}}^2 + I +\<\^{h}\>)\^{A}_{\^{\mf{p}}}} > \a-(d+1) \text{ because $\^{X} \cap \^{H}$ is singular at $\^{\mf{p}}$}
\end{split}
\end{align}

Now we have a surjective map $T^*_{\^{\mf{p}}}(X) \to T^*_{\^{\mf{p}}}(X \cap H)$, using the dimensions computed above we get that the map is an isomorphism. So we obtain that $\^{h}$ is linearly dependent on $f_1,\dots,f_k$ in $\^{A}_{\^{\mf{p}}}$ mod $\^{\mf{p}}^2$. We lift this dependence to $\^{A}$ to get an element $\^{t}$ not in $\^{\mf{p}}$ \st

\[\^{t}\big(\^{\g}\^{h} + \sum_i \^{\g_i} f_i\big) \in  \^{\mf{p}}^2\]

To ensure the dimension of the Artin approximation of $\^{S}$ is $d+1$, We will use the equations given by Artin in \cite[3.8]{ArtinApprox}. 
Let $\^{\mf{p}}$ be generated by $\<\^{g}_1,\ldots,\^{g}_l\>$.
Assume further $\^{g}_1,\dots,\^{g}_{n-d+1}$ form a regular sequence. Geometrically this means $\^{S}$ is a subscheme of a complete intersection $\^{C}$ (defined by $\<\^{g}_1,\ldots,\^{g}_{n-d+1}\>$) of dimension $d+1$. Then $\^{S}$ must be a component of $\^{C}$. 
Choose $\^{d} \in \^{A}$ \st $\^{d}$ vanishes on all components of $\^{C}$ except $\^{S}$. Then $\^{d}^n \^{g}_i \in \<\^{g}_1,\ldots,\^{g}_{n-d+1}\>$. We now look at the following ``equations". The statement ``belongs to an ideal", can be made into an equation by replacing the dependence relation with auxiliary variables. 

\[\^{h} \in \<\^{g}_1,\dots,\^{g}_l\>\qquad
f_i \in \<\^{g}_1,\dots,\^{g}_l\>\qquad
\^{t} \not\in \<\^{g}_1,\dots,\^{g}_l\>\]
\[\^{t}(\^{\g}\^{h} + \sum_i \^{\g_i} f_i) \in  \<\^{g}_1,\dots,\^{g}_l\>^2\]
\[\^{d}^n \^{g}_i \in \<\^{g}_1,\dots,\^{g}_{n-d+1}\>\]

Applying Artin approximation with $\^{h}$, $\^{g}_1,\dots,\^{g}_l$, $\^{d}$, $\^{t}$ as variables we get elements $h,g_1,\dots,g_l,d, t\in\ring{A}$ satisfying the above relations.
The last condition ensures that the dimension of scheme $\ring{S}$  defined by $\<g_1,\ldots,g_l\>$ is greater than $d$. If we approximate $\^{g}_i$to a large degree, then $\<g_1,\dots,g_l\>$ must be a prime. So, the rest of the conditions give us $\ring{S} \in \S_{\ring{X} \cap \ring{H}}$.
\end{proof}

We see that if the singular locus of $X \cap H_i$ has dimension $d$ and its multiplicity tends to infinity, then we have some hyperplane  \st $X \cap H$ has dimension $\geq d+1$. But does there exist some hyperplane \st $X \cap H$ has dimension exactly $d+1$? The answer is ``yes" as proved in the following Theorem, which can be thought of as the converse of Theorem \ref{Hyperplane section theorem}.

\begin{thm}
Let $X$ be an irreducible variety of $\Spec(\ring{A})$ (or $\Spec(A)$) \st $\dim(\S_X) \leq d$.
Let $H$ be a hyperplane \st $\S_{X\cap H}$ is of dimension $d>0$. Given $i \in \bb{Z}_{>0}$ and $j \in \bb{Z}$ \st $d-1 \geq j \geq \dim(\S_X) -1$, there exists a hyperplane $H_i$ \st $\S_{X\cap H_i}$ is of dimension $j$ and the multiplicity of $\S_{X\cap H_i}$ is $\geq i$.
\end{thm}
\begin{proof}
It suffices to prove the statement for $j=d-1$.

Choose a coordinate $x_1,\ldots , x_n,z$ for $\ring{A}$, \st $H$ is given by $z=0$.

Choose a general hyperplane $H'$ satisfying the constraints (1)-(4).
\begin{enumerate}
\item $\begin{aligned}[t]
    \frac{\doe h'}{\doe z} \neq 0
\end{aligned}$

Using the implicit function Theorem, this implies that $H'$ is given by $z= f(x_1,\dots,x_n)$. Define $H_k$ to be the hyperplane given by $z=f^k(x_1,\dots,x_n)$.
Fix $i \geq 2$

\item $\dim(\S_{X\cap H} \cap H_i) = d-1$ (i.e. $H_i$ does not contain any $d$ dimensional component of $\S_{X\cap H}$.)
\item $\dim(\S_X \cap H_i) \leq d-1$ (i.e. $H_i$ does not contain any $d$ dimensional component of $\S_{X}$, if any.)
\item $H_i \cap X - (\S_{X \cap H} \cup \S_X)$ should be smooth.
\end{enumerate}

Let $C' = \S_{X\cap H} \cap H'$ and $C_i = \S_{X\cap H} \cap H_i$. Since $\S_{X\cap H} \subset H$, which is given by $\{z=0\}$. As a consequence, $C'$ and $C_i$ in $\S_{X\cap H}$ are defined by $f(x_1,\dots,x_n) = 0$ and $f(x_1,\dots,x_n)^i = 0$ respectively. So $C'_{red} = (C_{i})_{red}$, which implies $\dim(C') = \dim(C_{i}) = d-1$.

By choice of $H_i$ (using (2,3,4)), $\dim(\S_{H_i \cap X}) \leq d-1$.

Let $S$ be given by $f^{i-1}(x_1,\dots,x_n)$ in $\Spec(\ring{A})$.
Note that $C_i = \S_{X\cap H} \cap S$ is contained inside $\S_{X\cap H_i}$ because modulo $f^{i-1}(x_1,\dots,x_n)$, the spaces $\S_{X\cap H}$ and $\S_{X\cap H_i}$ are the same. Therefore, multiplicity of $\S_{X\cap H_i}$ = multiplicity of $C_i$ which is greater than or equals to $i$.
\end{proof}

\renewcommand{\abstractname}{Acknowledgements}
\begin{abstract}  The first author would like to thank  Mathematisches Forschungsinstitut Oberwolfach for its  Research in Pairs program where this project was initiated with Mihai Tibar. He would like to express his gratitude to  Mihai Tibar for the same.

This work was supported by the Department of Atomic Energy, Government of India [project no. 12 - R\&D - TFR - 5.01 - 0500].
\end{abstract}
\renewcommand{\abstractname}{Abstract}
%\cite{GurjarHST}
%\cite{MoriHST}
%\cite{MatsumuraCRT}

\bibliographystyle{alpha}
\bibliography{references}

\end{document}